\documentclass[11pt,reqno]{amsart}
\usepackage{amssymb,amscd,amsfonts,amsbsy,amsmath,amsthm,amssymb}
\usepackage{mathrsfs}

\newtheorem{proposition}{Proposition}[section]
\newtheorem{theorem}[proposition]{Theorem}

\newtheorem{definition}[proposition]{Definition}

\theoremstyle{remark}

\newtheorem*{notationa}{Dimensions $d$ and $m$}
\newtheorem*{notationb}{Ambient Space $\mathscr{X}$}
\newtheorem*{notationc}{Subspace $E\subseteq\mathscr{X}$}
\newtheorem*{notationd}{Integral Operator $\Theta_E$}

\usepackage{color}

\begin{document}

\title[Square Function Estimates on Uniformly Rectifiable Sets]
{Square Function Estimates in Spaces of Homogeneous Type and on 
Uniformly Rectifiable Euclidean Sets}    
\footnotetext[1]{The work of the authors has been supported in part by 
the US NSF and the Simons Foundation}

\author[Hofmann]{Steve Hofmann}
\author[Mitrea]{Dorina Mitrea}
\author[Mitrea]{Marius Mitrea}
\author[Morris]{Andrew J. Morris}

\address{Steve Hofmann\\ University of Missouri\\ Columbia \\ MO 65211\\ USA}
\email{hofmanns@missouri.edu}

\address{Dorina Mitrea\\ University of Missouri\\ Columbia \\ MO 65211\\ USA}
\email{mitread@missouri.edu}

\address{Marius Mitrea\\ University of Missouri\\ Columbia \\ MO 65211\\ USA}
\email{mitream@missouri.edu}

\address{Andrew J. Morris\\ University of Missouri\\ Columbia \\ MO 65211\\ USA}
\email{morrisaj@missouri.edu}

\date{\today}

\begin{abstract}
We announce a local $T(b)$ theorem, an inductive scheme, and $L^p$ extrapolation 
results for $L^2$ square function estimates related to the analysis of integral 
operators that act on Ahlfors--David regular sets of arbitrary codimension in 
ambient quasi-metric spaces. The inductive scheme is a natural application of 
the local $T(b)$ theorem and it implies the stability of $L^2$ square function 
estimates under the so-called big pieces functor. In particular, this analysis 
implies $L^p$ and Hardy space square function estimates for integral operators 
on uniformly rectifiable subsets of the Euclidean space. 
\end{abstract}

\maketitle

\section{Introduction}\label{Sect:1}

This work is motivated by the $L^2$ square function estimates 
proved by G.~David and S.~Semmes in \cite{DaSe91}, \cite{DaSe93},
for convolution type integral operators associated with the Riesz kernels 
$x_j/|x|^{n+1}$, $1\leq j\leq n+1$, on uniformly rectifiable sets in 
${\mathbb{R}}^{n+1}$. Our main aim here is to extend their results in a 
number of directions, including the consideration of a larger class of 
kernels, and establishing $L^p$ as well as Hardy type square function estimates.
A substantial portion of our analysis is valid in the general setting of abstract 
quasi-metric spaces (which automatically forces the consideration of 
non-convolution type kernels), and we succeed in dealing with Ahlfors-David 
regular sets of arbitrary codimension in that general setting. 
We announce three results obtained in this setting. The starting point is a 
local $T(b)$ theorem for square functions which is then used to prove an inductive 
scheme whereby square function estimates are shown to be stable under the 
so-called big pieces functor. The third result is an extrapolation principle 
whereby an $L^p$ (or weak-$L^p$, or Hardy space $H^p$) square function estimate
for one $p$ yields a full range of square function bounds. 
We also indicate how the inductive scheme can be applied to obtain 
square function estimates for a large class of integral operators that act 
on uniformly rectifiable sets of codimension one in Euclidean space. 

To formulate the latter result, we need to introduce some notation and recall 
some definitions. Following \cite{DaSe91}, \cite{DaSe93}, a closed set ${\Sigma}\subseteq{\mathbb{R}}^{n+1}$ is called {\tt uniformly} 
{\tt rectifiable} if it is $n$-dimensional Ahlfors-David regular 
and has big pieces of Lipschitz images. The property of being 
$n$-{\tt dimensional} {\tt Ahlfors-David} {\tt regular} (ADR for short) means 
that there exists a constant $C\in[1,\infty)$ such that
\begin{align}\label{Q3HFiii}
C^{-1}\,r^n\leq{\mathcal{H}}^n\bigl(\Sigma\cap B(x,r)\bigr)\leq C\,r^n,
\quad\forall\,x\in {\Sigma},\,\,\,\,\mbox{for every finite }\,r
\in(0,{\rm diam}(\Sigma)],
\end{align}
where $B(x,r):=\big\{y\in{\mathbb{R}}^{n+1}:\,|x-y|<r\big\}$, 
${\rm diam}(\Sigma):=\sup\,\bigl\{|x-y|:\,x,y\in\Sigma\bigr\}\in(0,\infty]$, 
and ${\mathcal{H}}^n$ denotes the $n$-dimensional Hausdorff measure in 
${\mathbb{R}}^{n+1}$. The property of having {\tt big} {\tt pieces} 
{\tt of} {\tt Lipschitz} {\tt images} means that there exist constants 
$\eta,C\in(0,\infty)$ such that for 
each $x\in{\Sigma}$ and each finite $r\in(0,{\rm diam}(\Sigma)]$, there is a 
ball $B^{n}_r$ of radius $r$ in ${\mathbb{R}}^{n}$ and a Lipschitz map 
$\varphi:B^{n}_r\rightarrow {\mathbb{R}}^{n+1}$ with Lipschitz constant 
at most equal to $C$, such that
\begin{align}\label{3.1.9aS}
{\mathcal{H}}^{n}\bigl({\Sigma}\cap B(x,r)\cap\varphi(B^{n}_r)\bigr)\geq\eta\,r^n.
\end{align}
We now state our principal result in the Euclidean context, dealing with 
Hardy and $L^p$ square function estimates on uniformly rectifiable sets.

\begin{theorem}\label{UR-rest}
Suppose that $K$ is a real-valued function with the following properties:
\begin{equation}\label{K-BIS}
K\in C^2({\mathbb{R}}^{n+1}\setminus\{0\}),\quad
K\mbox{ is odd},\quad K(\lambda x)=\lambda^{-n}K(x)
\quad\forall\,\lambda>0,\,\,\,\forall\,x\in{\mathbb{R}}^{n+1}\setminus\{0\}.
\end{equation}
Let $\Sigma$ be a uniformly rectifiable subset of ${\mathbb{R}}^{n+1}$ and  
let $\sigma:={\mathcal{H}}^n\lfloor\Sigma$ denote the restriction 
of the $n$-dimensional Hausdorff measure in ${\mathbb{R}}^{n+1}$ to $\Sigma$.
For each $p\in\bigl(\frac{n}{n+1},\infty\bigr)$ let $H^p(\Sigma,\sigma)$ stand 
for the Lebesgue scale $L^p(\Sigma,\sigma)$ if $p\in(1,\infty)$, and the 
Coifman-Weiss scale of Hardy spaces on the space of homogeneous 
type $(\Sigma,|\cdot-\cdot|,\sigma)$ if $p\in\bigl(\frac{n}{n+1},1\bigr]$. 
Finally, consider the integral operator ${\mathcal{T}}$ acting on functions 
$f\in L^p(\Sigma,\sigma)$, $p\in(1,\infty)$, according to  
\begin{align}\label{T-BIS.2a}
{\mathcal{T}}f(x):=\int_\Sigma K(x-y)f(y)\,d{\sigma}(y),\qquad 
\forall\,x\in{\mathbb{R}}^{n+1}\setminus \Sigma.
\end{align}
Then the operator ${\mathcal{T}}$ extends to $H^p(\Sigma,\sigma)$ and there 
exists $C\in(0,\infty)$ such that
\begin{align}\label{SF-BIS-10}
\left\|\Bigl(\int_{\Gamma_{\kappa}(x)}|(\nabla{\mathcal{T}} f)(y)|^2\,
\frac{dy}{{\rm dist}\,(y,\Sigma)^{n-1}}\Bigr)^{\frac{1}{2}}
\right\|_{L^p_x(\Sigma,\sigma)}\!\!\!\leq C\|f\|_{H^p(\Sigma,\sigma)},
\quad\forall\,f\in H^p(\Sigma,\sigma),
\end{align}
where $\Gamma_\kappa(x):=\bigl\{y\in{\mathbb{R}}^{n+1}\setminus {\Sigma}:\,
|x-y|<(1+\kappa)\,{\rm dist}\,(y,\Sigma)\bigr\}$, for each $x\in {\Sigma}$.

As a corollary, the following $L^2$ square function estimate holds
\begin{align}\label{SF-BIS-ann}
\int_{{\mathbb{R}}^{n+1}\setminus \Sigma}|(\nabla{\mathcal{T}}f)(x)|^2
\,{\rm dist}\,(x,\Sigma)\,dx\leq C\int_\Sigma|f(x)|^2\,d\sigma(x),
\quad\forall\,f\in L^2(\Sigma,\sigma).
\end{align}
\end{theorem}

The fact that \eqref{SF-BIS-ann} is a consequence of \eqref{SF-BIS-10} with $p=2$
follows from Fubini's theorem and the ADR property of $\Sigma$. This being said, 
our proof of \eqref{SF-BIS-10} begins by first establishing \eqref{SF-BIS-ann}
and then deriving \eqref{SF-BIS-10} using arguments akin to Calder\'on-Zygmund 
theory which, in effect, indicate that deriving Hardy and $L^p$ 
square-function estimates is not unlike proving boundedness results 
for singular integral operators. That \eqref{SF-BIS-ann} holds in 
the case when ${\mathcal{T}}$ is associated as in \eqref{T-BIS.2a} with 
each of the Riesz kernels $K_j(x):=x_j/|x|^{n+1}$, $1\leq j\leq n+1$, is due 
to David and Semmes \cite{DaSe93}. Here we generalize their result in two basic 
aspects, by considering any kernel satisfying the purely real-variable 
conditions in \eqref{K-BIS} (which are not tied to any 
particular partial differential operator, in the manner that the kernels $K_j(x):=x_j/|x|^{n+1}$, $1\leq j\leq n+1$, are related to the Laplacian), 
and by considering the entire Hardy-Lebesgue scale (and not just $L^2$). 
In fact, it is also possible to obtain a version of Theorem~\ref{UR-rest} for variable 
coefficient kernels, which ultimately applies to integral operators 
on manifolds that are associated with the Schwartz kernels of certain classes of pseudodifferential operators acting between vector bundles, although we shall 
not discuss this further here. We sketch the proof of Theorem~\ref{UR-rest} 
in Sections~\ref{SS2.2}-\ref{SS2.3}.

It is both useful and instructive to separate the portion of the 
proof of Theorem~\ref{UR-rest}
which makes essential use of Euclidean tools from the portion which works in the
general context of quasi-metric spaces. As regards the first portion, two ingredients
are essential for our approach, namely the fact that the square function estimate
\eqref{SF-BIS-ann} holds when $\Sigma$ is a Lipschitz graph in ${\mathbb{R}}^{n+1}$,
and the recent characterization of uniformly rectifiable subsets of
${\mathbb{R}}^{n+1}$ proved by J.~Azzam and R.~Schul in \cite{AS} in terms of 
the two-fold iteration of the so-called big pieces functor starting with Lipschitz
graphs. Concerning the second portion of the proof, mentioned earlier, it is 
remarkable that a large number of significant results, including a local $T(b)$
theorem for square functions, a geometrically inductive scheme showing that 
$L^2$ square function estimates are stable under the action of the big pieces functor,
and extrapolation results for square function estimates, make no specific use of the
vector space structure of the ambient space (as well as any other beneficial aspect which
such a setting entails). As such, the aforementioned results can be developed 
abstractly in the general context of quasi-metric spaces. This point of view 
is systematically pursued in Section~\ref{Sect:2}.

\section{Main Results}\label{Sect:2}

We begin by discussing the abstract setting which constitutes the foundation
of most of our results in this paper. The following notation and assumptions 
are fixed and used without specific reference hereafter, unless otherwise specified.

\vspace{6pt}

\begin{notationa}
Let $d$ and $m$ denote two real numbers such that $0<d<m<\infty$.\vspace{6pt}
\end{notationa}

\begin{notationb}
Let $({\mathscr{X}},\rho,\mu)$ denote an $m$-dimensional 
Ahlfors-David regular ($m$-ADR) space. This is defined to mean that 
${\mathscr{X}}$ is a set of cardinality at least two, equipped with a 
quasi-distance $\rho$, and a Borel measure $\mu$ with the property 
that all $\rho$-balls are $\mu$-measurable, and for which there exists 
a constant $C\in[1,\infty)$ such that
\begin{align}\label{Q3HF}
C^{-1}\,r^m\leq\mu\bigl(B_\rho(x,r)\bigr)\leq C\,r^m,
\quad\forall\,x\in {\mathscr{X}},\,\,\,\,\mbox{for every finite }\,r
\in(0,{\rm diam}_\rho({\mathscr{X}})],
\end{align}
where $B_\rho(x,r) := \{y\in \mathscr{X} : \rho(x,y)<r\}$, and 
${\rm diam}_\rho({\mathscr{X}}) :=\sup\,\bigl\{\rho(x,y):\,x,y\in \mathscr{X}\bigr\}\in(0,\infty]$. The constant $C$ in \eqref{Q3HF} 
will be referred to as the ADR constant of ${\mathscr{X}}$. The quasi-distance $\rho:{\mathscr{X}}\times{\mathscr{X}}\to[0,\infty)$ has the property that 
there exist two constants $\widetilde{C}_\rho,C_\rho\in[1,\infty)$ such that
\begin{align}\label{gabn-T.2-ann}
\rho(x,y)=0\Leftrightarrow x=y,\quad
\rho(y,x)\leq \widetilde{C}_\rho\,\rho(x,y),\quad
\rho(x,y)\leq C_\rho\max\{\rho(x,z),\rho(z,y)\},
\end{align}
for all $x,y,z\in\mathscr{X}$, and we introduce the related constant
\begin{align}\label{Cro-ann}
\alpha_\rho:=\frac{1}{\log_2 C_\rho}\in(0,\infty].
\end{align}
The measure $\mu$ is Borel with respect to topology $\tau_\rho$ canonically 
induced by $\rho$, which is defined to be the largest topology on 
${\mathscr{X}}$ with the property that for each point $x\in{\mathscr{X}}$ 
the family $\{B_\rho(x,r)\}_{r>0}$ is a fundamental system of 
neighborhoods of $x$.
\end{notationb}

\begin{notationc}
Let $(E,\sigma)$ consist of a closed subset $E$ of $({\mathscr{X}},\tau_\rho)$ 
and a Borel regular measure~$\sigma$ on $(E,\tau_{\rho|_{E}})$ with the 
property that $(E,\rho\bigl|_E,\sigma)$ is a $d$-dimensional {\rm ADR} space, 
where $\rho|_{E}$ denotes the restriction of $\rho$ to $E\times E$. 
From \cite{MMMM-G} (which extends work in \cite{MaSe79}) we know that, in this context, 
there exists a symmetric quasi-distance $\rho_{\#}$ on ${\mathscr{X}}$, 
called the regularization of $\rho$, such that $\rho_{\#}$ and $\rho$ are 
equivalent quasi-metrics, the topology $\tau_{\rho_{\#}}=\tau_{\rho}$, 
and the regularized distance
\begin{align}\label{REG-DDD}
\delta_E(x):=\inf\,\big\{{\rho_{\#}}(x,y):\,y\in E\big\},
\quad\forall\,x\in{\mathscr{X}},
\end{align}
is continuous on $({\mathscr{X}},\tau_\rho)$.

It is also well-known (see \cite{Christ}, \cite{David1988}) that in this context 
there exists a dyadic cube structure on~$E$. In particular, fix $\kappa_E$ in ${\mathbb{Z}}\cup\{-\infty\}$ with the property that 
$2^{-\kappa_E-1}< {\rm diam}_\rho(E)\leq 2^{-\kappa_E}$. For each integer $k\geq\kappa_E$, a collection ${\mathbb{D}}_k(E):=\{Q_\alpha^k\}_{\alpha\in I_k}$ 
of subsets $Q_\alpha^k$ of $E$, indexed by a nonempty and at most countable 
set $I_k$ of indices $\alpha$, is fixed such that the entire collection 
\begin{align}\label{gcEd}
{\mathbb{D}}(E):=\bigcup_{k\in{\mathbb{Z}},\,k\geq\kappa_E}{\mathbb{D}}_k(E)
\end{align}
has properties analogous to the ordinary dyadic cube structure of $\mathbb{R}^n$. 
We refer to the sets $Q$ in ${\mathbb{D}}_k(E)$ as dyadic cubes with side length $l(Q):=2^{-k}$. We stress that all quantitative aspects pertaining to 
${\mathbb{D}}(E)$ are controlled in terms of the {\rm ADR} constants of $E$ 
(as well as on ${\rm diam}_\rho(E)$ when $E$ is bounded).

There also exists a Whitney covering of $\mathscr{X}\setminus E$ 
(see for instance \cite{MMMM-G}), since $(\mathscr{X},\rho)$ is geometrically 
doubling and $E$ is closed in $(\mathscr{X},\tau_\rho)$. We use this covering to 
associate to each cube $Q$ in ${\mathbb{D}}(E)$ the region
$\mathcal{U}_Q\subseteq\mathscr{X}\setminus E$ which is the union of all Whitney 
cubes of size comparable to $\ell(Q)$ located at a distance less than or equal to
 $C\ell(Q)$ from $Q$. 
Intuitively, the reader should think of these as being analogous to the upper halves 
of ordinary Carleson regions in the Euclidean upper half-space. In particular, 
the dyadic Carleson tent $T_E(Q)$ over $Q$ (relative to the set $E$) 
may now be defined as
\begin{align}\label{gZSZ-3}
T_E(Q):=\bigcup_{Q'\in{\mathbb{D}}(E),\,\,Q'\subseteq Q}{\mathcal{U}}_{Q'}.
\end{align}
\end{notationc}

\begin{notationd}
Let ${\theta}:(\mathscr{X}\times\mathscr{X})
\setminus\{(x,x):\,x\in \mathscr{X}\}\longrightarrow{{\mathbb{R}}}$
denote a Borel measurable function, with respect to the product topology 
$\tau_\rho\times\tau_\rho$, for which there exist finite positive constants 
$C_{\theta}$, $\alpha$, $\upsilon$ such that for all $x,y\in\mathscr{X}$ 
with $x\neq y$ the following hold:
\begin{eqnarray}\label{hszz-A-ann}
&& \mbox{[{\tt decay condition:}]}\quad
|{\theta}(x,y)|\leq\frac{C_{\theta}}{\rho(x,y)^{d+\upsilon}},
\\[4pt]
&& \mbox{[{\tt H\"older regularity:}]}\quad
|{\theta}(x,y)-{\theta}(x,\widetilde{y})|\leq C_{\theta} \frac{\rho(y,\widetilde{y})^\alpha}{\rho(x,y)^{d+\upsilon+\alpha}}, 
\nonumber\\[4pt]
&&\hskip 1.70in
\forall\,\widetilde{y}\in\mathscr{X}\setminus\{x\}\,\,\mbox{ with }\,\,
\rho(y,\widetilde{y})\leq\tfrac{1}{2}\rho(x,y).\quad
\label{hszz-3-A-ann}
\end{eqnarray}
The integral operator $\Theta_E$ is then defined for all functions 
$f\in L^p(E,\sigma)$, $1\leq p\leq\infty$, by
\begin{align}\label{operator-A-ann}
(\Theta_E f)(x):=\int_E {\theta}(x,y)f(y)\,d\sigma(y),
\qquad\forall\,x\in\mathscr{X}\setminus E.
\end{align}
\end{notationd}

We proceed to describe our main tools in the treatment of square function estimates.

\subsection{An arbitrary codimension local $T(b)$ theorem for square functions}
\label{SS2.1}

The local $T(b)$ theorem below states that a global square function estimate for 
the integral operator $\Theta_E$ holds if there exists a family of suitably 
non-degenerate and normalized functions $\{b_Q\}_{Q\in{\mathbb{D}}(E)}$ with 
the property that, for each $Q\in{\mathbb{D}}(E)$, a uniform, scale-invariant, 
local version of the $L^2$ square function estimate on the dyadic Carleson tent 
$T_E(Q)$ holds for $\Theta_E$ acting on $b_Q$. Naturally, the formulation of 
the aforementioned square function estimates takes into account both the 
co-dimension $m-d$ of $E$ in ${\mathscr{X}}$, and the exponent $\upsilon$ 
intervening in the decay condition \eqref{hszz-A-ann}.

Here is the formal statement of our first main result which generalizes a 
Euclidean codimension one version that was implicit in the solution of the 
Kato problem in \cite{AHLMcT}, \cite{HLMc}, \cite{HMc}, and later formulated 
explicitly in \cite{Au}, \cite{Ho3}, \cite{HMc2}. 

\begin{theorem}\label{Thm:localTb}
If there exist two constants $C_0\in[1,\infty)$, $c_0\in(0,1]$, and a 
collection $\{b_Q\}_{Q\in{\mathbb{D}}(E)}$ of $\sigma$-measurable functions $b_Q:E\rightarrow{\mathbb{C}}$ such that for each $Q\in{\mathbb{D}}(E)$ 
the following hold:
\begin{enumerate}
\item[(1)] $\int_E |b_Q|^2\,d\sigma\leq C_0\sigma(Q)$;
\item[(2)] there exists $\widetilde{Q}\in{\mathbb{D}}(E)$, $\widetilde{Q}\subseteq Q$,
$\ell(\widetilde{Q})\geq c_0\ell(Q)$, and 
$\left|\int_{\widetilde{Q}}b_Q\,d\sigma\right|\geq\frac{1}{C_0}\,\sigma(\widetilde{Q})$;
\item[(3)] $\int_{T_E(Q)}|(\Theta\,b_Q)(x)|^2
\delta_E(x)^{2\upsilon-(m-d)}\,d\mu(x)\leq C_0\sigma(Q)$,
\end{enumerate}
then there exists a constant $C\in(0,\infty)$, depending only on $C_0$, 
$C_{\theta}$, and the {\rm ADR} constants of $E$ and ${\mathscr{X}}$, 
as well as on ${\rm diam}_\rho(E)$ when $E$ is bounded, such that
\begin{align}\label{G-UF-2}
\int_{\mathscr{X}\setminus E}\big|(\Theta_E f)(x)\big|^2
\delta_E(x)^{2\upsilon-(m-d)}\,d\mu(x)\leq C\int_E|f(x)|^2\,d\sigma(x),
\end{align}
for all $f\in L^2(E,\sigma)$.
\end{theorem}

\subsection{An inductive scheme for square function estimates}\label{SS2.2}

The inductive scheme in the theorem below shows that the integral operator 
$\Theta_E$ satisfies square function estimates whenever the set $E$ contains 
(uniformly, at all scales and locations) so-called big pieces of sets on which 
square function estimates hold. In short, we say that big pieces of square 
function estimates (BPSFE) imply square function estimates (SFE). We sketch the 
proof of this result, since it is a natural application of our local $T(b)$ theorem, 
and then indicate how the inductive scheme can be applied to uniformly rectifiable
sets by proving \eqref{SF-BIS-ann} in Theorem~\ref{UR-rest}.

The formulation of this result requires the notion of Hausdorff measure 
in the quasi-metric setting. Specifically, 
let $\mathscr{H}_{{\mathscr{X}}\!,\,\rho_{\#}}^d$ denote the $d$-dimensional 
Hausdorff measure on $(\mathscr{X},\rho_{\#})$
(see \cite[Definition 4.70]{MMMM-G}), and for any closed subset $A$ of $(\mathscr{X},\tau_\rho)$, let $\mathscr{H}_{{\mathscr{X}}\!,\,\rho_{\#}}^d\lfloor A$ denote the measure given by the restriction of $\mathscr{H}_{{\mathscr{X}}\!,\,\rho_{\#}}^d$ to $A$. 
We begin by defining what it means for the set $E$ to have 
big pieces of square function estimate.

\begin{definition}\label{sjvs-ann}
The set $E\subseteq{\mathscr{X}}$ is said to have big pieces of square function 
estimate (BPSFE) relative to the kernel $\theta$ if there exist three constants $\eta,C_1,C_2\in(0,\infty)$ with the property that for each ${Q\in{\mathbb{D}}(E)}$ 
there exists a closed subset $E_Q$ of $(\mathscr{X},\tau_\rho)$ such that $\bigl(E_Q,\rho\bigl|_{E_Q},
{\mathscr{H}}_{{\mathscr{X}}\!,\,\rho_{\#}}^d\lfloor E_{Q} \bigr)$ is a 
$d$-dimensional ADR space with ADR constant less than or equal to $C_1$, 
and which satisfies
\begin{align}\label{yvg2-ann}
{\mathscr{H}}_{{\mathscr{X}}\!,\,\rho_{\#}}^d(E_Q\cap Q)
\geq\eta\,{\mathscr{H}}_{{\mathscr{X}}\!,\,\rho_{\#}}^d(Q)
\end{align}
as well as
\begin{align}\label{avhai2-ann}
\begin{array}{c}
\displaystyle
\int_{\mathscr{X}\setminus E_Q}|\Theta_{E_Q} f(x)|^2\,
{\rm dist}_{\rho_{\#}}(x,E_Q)^{2\upsilon-(m-d)}\,d\mu(x)
\leq C_2\int_{E_Q} |f|^2\ d{\mathscr{H}}_{{\mathscr{X}}\!,\,\rho_{\#}}^d\lfloor E_{Q},
\end{array}
\end{align}
for all $f\in L^2(E_Q,{\mathscr{H}}_{{\mathscr{X}}\!,\,\rho_{\#}}^d\lfloor E_{Q})$, 
where $\Theta_{E_Q}$ is the operator associated with $E_Q$ as in \eqref{operator-A-ann}.
The constants $\eta,C_1,C_2$ will collectively be referred to as the BPSFE character 
of the set $E$.
\end{definition}

We now state and sketch the proof of our second main result. 

\begin{theorem}\label{Thm:BPSFtoSF}
If the set $E\subseteq{\mathscr{X}}$ has {\rm BPSFE} relative to the kernel 
$\theta$, then there exists a constant $C\in(0,\infty)$, depending only on 
$\rho$, $m$, $d$, $\upsilon$, $C_{\theta}$, the {\rm BPSFE} character of $E$, and the {\rm ADR} constants of $E$ and ${\mathscr{X}}$, such that 
\begin{align}\label{vlnGG-ann}
\int_{\mathscr{X}\setminus E}|\Theta_E f(x)|^2\,
\delta_E(x)^{2\upsilon-(m-d)}\,d\mu(x) 
\leq C\int_E|f|^2\,d{\mathscr{H}}_{{\mathscr{X}}\!,\,\rho_{\#}}^d,
\end{align}
for all $f\in L^2(E,{\mathscr{H}}_{{\mathscr{X}}\!,\,\rho_{\#}}^d\lfloor E)$.
\end{theorem}

\begin{proof}
For each $Q\in{\mathbb{D}}(E)$ there exists $E_Q\subseteq\mathscr{X}$ satisfying \eqref{yvg2-ann}-\eqref{avhai2-ann}, since $E$ has BPSFE relative to the 
kernel~$\theta$. We then define the function $b_Q:E\rightarrow{\mathbb{R}}$ 
by setting
\begin{align}\label{BBss-ann}
b_Q(y):={\mathbf{1}}_{Q\cap E_Q}(y),\qquad\forall\,y\in E.
\end{align}
The strategy for proving \eqref{vlnGG-ann} is to invoke Theorem~\ref{Thm:localTb}
for the family $\{b_Q\}_{Q\in{\mathbb{D}}(E)}$, and as such, it suffices to verify conditions (1)--(3) in Theorem~\ref{Thm:localTb}. Condition (1) is immediate, 
and condition (2) with $\widetilde{Q}:=Q$ is a consequence of \eqref{yvg2-ann}. 
To verify condition (3), we introduce a constant $C_A\in(1,\infty)$ to be chosen 
later, and the set
\begin{equation}
A:=\big\{x\in{\mathscr{X}}:\,
C_A^{-1}\delta_E(x)\leq\delta_{E_Q}(x)\leq C_A\delta_E(x)\big\},
\end{equation}
in order to write
\begin{align}\label{KLGB}
\int_{T_E(Q)}\big|\Theta_E & b_Q(x)\big|^2\delta_E(x)^{2\upsilon-(m-d)}\,d\mu(x) 
=I_{\mathscr{X}\setminus A}+I_A,
\end{align}
where 
\begin{eqnarray}\label{KLGB.222}
I_{\mathscr{X}\setminus A}:=\int_{T_E(Q)\setminus A} 
\big|\Theta_E b_Q(x)\big|^2 \,\delta_E(x)^{2\upsilon-(m-d)}\,d\mu(x),
\\[4pt]
I_A:=\int_{T_E(Q)\cap A}\big|\Theta_E b_Q(x)\big|^2 \,\delta_E(x)^{2\upsilon-(m-d)}
\,d\mu(x).
\label{KLGB.333}
\end{eqnarray}
The estimate for $I_{\mathscr{X}\setminus A}$ requires choosing $C_A$ 
sufficiently large, in accordance with the ADR geometry of~$E$ and the 
decay of the kernel~$\theta$. The idea is to rely on a pointwise bound 
for $\Theta_E b_Q$ and Carleson measure estimates of a purely geometric nature.
More specifically, we prove that a judicious choice of $C_A$ gives 
\begin{eqnarray}\label{gwo}
\int_{\{x\in T_E(Q):\,\delta_{E_Q}(x)>C_A\delta_E(x)\}}
\delta_{E_Q}(x)^{-2\upsilon}\delta_E(x)^{2\upsilon-(m-d)}\,d\mu(x)\leq C\sigma(Q),
\end{eqnarray}
plus a Carleson measure estimate similar in nature (involving a suitable 
choice of the powers of the distance functions) on the piece 
$\{x\in T_E(Q):\,\delta_{E_Q}(x)<C_A^{-1}\delta_E(x)\}$. As regards 
$I_{A}$, we use~\eqref{avhai2-ann} to obtain (with ${\bf 1}_A$ denoting the 
characteristic function of $A$)
\begin{align}\label{KLGB.3}
\begin{split}
I_A &=\int_{T_E(Q)\setminus E_Q}\big|\Theta_{E_Q} b_Q(x)\big|^2{\mathbf{1}}_A(x)\,
\delta_E(x)^{2\upsilon-(m-d)}\,d\mu(x)
\\[4pt]
&\lesssim\int_{\mathscr{X}\setminus E_Q}\big|\Theta_{E_Q}b_Q(x)\big|^2\,
\delta_{E_Q}(x)^{2\upsilon-(m-d)}\,d\mu(x)
\\[4pt]
&\lesssim\int_{E_Q}|b_Q|^2\,d{\mathscr{H}}_{{\mathscr{X}}\!,\,\rho_{\#}}^d\lfloor E_{Q} \lesssim {\mathscr{H}}_{{\mathscr{X}}\!,\,\rho_{\#}}^d(Q),
\end{split}
\end{align}
as desired.
\end{proof}

We now prove \eqref{SF-BIS-ann} in Theorem~\ref{UR-rest} by combining 
Theorem~\ref{Thm:BPSFtoSF} with a characterization of uniform rectifiability 
obtained recently by J.~Azzam and R.~Schul in \cite{AS}.

\begin{proof}[Proof of \eqref{SF-BIS-ann} in Theorem~\ref{UR-rest}]
To get started recall that the set ${\Sigma}$ is uniformly rectifiable, so by 
the characterization of such sets in~\cite[Corollary~1.7]{AS}, it follows that 
$\Sigma$ has big pieces of big pieces of Lipschitz graphs (BP)$^2$LG. The next step 
is to prove estimate \eqref{SF-BIS-ann} in the case when $E$ is a Lipschitz graph
in ${\mathbb{R}}^{n+1}$ and this is achieved by building on earlier work in \cite{CMcM},~\cite{Ho},~\cite{HL}. We conclude that $\Sigma$ has big pieces 
of big pieces of square function estimates, i.e., (BP)$^2$SFE. We shall not 
define (BP)$^2$LG nor (BP)$^2$SFE here, but both should be understood in a natural 
manner. Theorem~\ref{Thm:BPSFtoSF} can be 
iterated to show that (BP)$^2$SFE implies BPSFE, which in turn implies that 
the square function estimate in \eqref{vlnGG-ann} holds, so \eqref{SF-BIS-ann} follows 
by applying \eqref{vlnGG-ann} with $(d,m,\mathscr{X},E,\theta,\Theta)
=(n,n+1,\mathbb{R}^{n+1},\Sigma,\nabla K,\nabla\mathcal{T})$.
\end{proof}

\subsection{$L^p$ square function estimates}\label{SS2.3}

We now consider $L^p$ versions of the $L^2$ square function estimates 
considered above for the integral operator $\Theta_E$. We list three 
extrapolation theorems and show how they can be applied to uniformly 
rectifiable sets by proving \eqref{SF-BIS-10} in Theorem~\ref{UR-rest}.
The first states that $L^2$ square function 
estimates follow automatically from weak-$L^p$ square function estimates for 
any $p\in(0,\infty)$. The second and third provide a range of sufficient conditions 
for $L^p$, weak-$L^p$, and Hardy space $H^p$ square function estimates to hold.

For any $\kappa>0$, consider the nontangential approach region
\begin{align}\label{TLjb}
\Gamma_\kappa(x):=\bigl\{y\in{\mathscr{X}}\setminus E:\,
\rho_{\#}(x,y)<(1+\kappa)\,\delta_E(y)\bigr\},
\qquad\forall\,x\in E.
\end{align}
The following is the first extrapolation theorem.

\begin{theorem}\label{VGds-L2XXX-ann} 
If $\kappa,p,C_o$ are finite positive constants such that for every 
$z\in E$ and $r>0$ the surface ball ${\Delta:=E\cap B_{\rho_{\#}}(z,r)}$ satisfies
\begin{align}\label{dtbh-L2iii}
\sigma\left(\Bigl\{x\in E:\,
\int_{\Gamma_{\kappa}(x)}|(\Theta_E{\mathbf{1}}_{\Delta})(y)|^2
\,\delta_E(y)^{2\upsilon-m}\,d\mu(y)>\lambda^2\Bigr\}\right)
\leq C_o\lambda^{-p}\sigma(\Delta),\quad\forall\,\lambda>0,
\end{align}
then there exists $C\in(0,\infty)$ which depends only on $\kappa,p,C_o$ 
and finite positive background constants (including ${\rm diam}_\rho(E)$ 
in the case when $E$ is bounded) such that 
\begin{align}\label{k-tSSiii-ann}
\int_{\mathscr{X}\setminus E}
|(\Theta_E f)(x)|^2\delta_E(x)^{2\upsilon-(m-d)}\,d\mu(x)
\leq C\int_E|f(x)|^2\,d\sigma(x),\qquad\forall\,f\in L^2(E,\sigma).
\end{align}
\end{theorem}

The requirement in \eqref{dtbh-L2iii} is less restrictive than the 
standard weak-$L^p$ estimate 
\begin{align}\label{eqrem-ann}
\sup_{\lambda>0}\left[\lambda\cdot
\sigma\Bigl(\Bigl\{x\in E:\int_{\Gamma_{\kappa}(x)}|(\Theta_E f)(y)|^2
\delta_E(y)^{2\upsilon-m}\,d\mu(y)>\lambda^{2}\Bigr\}\Bigr)^{1/p}\right]
\leq C_o\|f\|_{L^{p}(E,\sigma)}
\end{align}
for every $f\in L^p(E,\sigma)$, since \eqref{dtbh-L2iii} follows by 
substituting $f={\mathbf{1}}_{\Delta}$ in \eqref{eqrem-ann}.

The next extrapolation theorem shows that a weak-$L^q$ square function estimate 
for any $q\in(0,\infty)$ implies that square functions are bounded from $H^p$ 
into $L^p$ for all $p\in(\frac{d}{d+\gamma},\infty)$, where $d$ is the 
dimension of $E$ and (recalling $\alpha_\rho$ from \eqref{Cro-ann} and 
$\alpha$ from \eqref{hszz-3-A-ann}) the constant
\begin{align}\label{WQ-tDD-ann}
\gamma:=\min\,\bigl\{\alpha_\rho,\alpha\bigr\}.
\end{align}
The theory of Hardy-Lebesgue spaces $H^p=H^p(E,\rho\bigl|_E,\sigma)$ 
for ADR subsets of a quasi-metric space has been developed in~\cite{MMMM-G}
for $p$ belonging to an interval containing $\big(\frac{d}{d+\alpha_\rho},\infty\big)$.
These spaces become $L^p(E,\sigma)$ when $p\in(1,\infty)$, and in the case when 
$p\in\big(\frac{d}{d+\alpha_\rho},1\big]$ they have an atomic characterization 
as in the work of R.~Coifman and G.~Weiss in \cite{CoWe77}, as well as a 
grand maximal function characterization akin to that established by 
R.~Mac\'{i}as and C.~Segovia in~\cite{MaSe79II}.

\begin{theorem}\label{VGds-2-ann}
Fix $\kappa>0$. Given $q\in(1,\infty)$ and $p\in\bigl(\frac{d}{d+\gamma},\infty\bigr)$, consider the estimate 
\begin{align}\label{kt-Dc-ann}
\left\|\Bigl(\int_{\Gamma_{\kappa}(x)}|(\Theta_E f)(y)|^q\,
\frac{d\mu(y)}{\delta_E(y)^{m-q\upsilon}}\Bigr)^{\frac{1}{q}}
\right\|_{L^p_x(E,\sigma)}\!\!\!
\leq C\|f\|_{H^p(E,\rho|_{E},\sigma)},\quad\forall\,f\in H^p(E,\rho|_{E},\sigma),
\end{align}
for some constant $C\in(0,\infty)$.
\begin{enumerate}
\item[(I)] Assume that $q\in(1,\infty)$ has the property that, for some constant $C_o\in(0,\infty)$, either 
\begin{align}\label{kt-Dc-BIS-ann}
\left\|\Bigl(\int_{\Gamma_{\kappa}(x)}|(\Theta_E f)(y)|^q\,
\frac{d\mu(y)}{\delta_E(y)^{m-q\upsilon}}\Bigr)^{\frac{1}{q}}
\right\|_{L^q_x(E,\sigma)}\!\!\!
\leq C_o\|f\|_{L^q(E,\sigma)},\quad\forall\,f\in L^q(E,\sigma),
\end{align}
or there exists $p_o\in(q,\infty)$ such that for every 
$f\in L^{p_o}(E,\sigma)$ there holds
\begin{align}\label{dtbjHT-ann}
\sup_{\lambda>0}\left[\lambda\cdot
\sigma\Bigl(\Bigl\{x\in E:\int_{\Gamma_{\kappa}(x)}|(\Theta_E f)(y)|^q\,
\frac{d\mu(y)}{\delta_E(y)^{m-q\upsilon}}>\lambda^{q}\Bigr\}\Bigr)^{1/p_o}\right]
\leq C_o\|f\|_{L^{p_o}(E,\sigma)}.
\end{align}
Then \eqref{kt-Dc-ann} holds for every $p\in\bigl(\frac{d}{d+\gamma},\infty\bigr)$.
\item[(II)] Assume that $q\in(1,\infty)$ is such that there exist $p_o\in(1,\infty)$ 
and a constant $C_o\in(0,\infty)$ such that \eqref{dtbjHT-ann} holds for every 
$f\in L^{p_o}(E,\sigma)$. Then \eqref{kt-Dc-ann} holds for every $p\in(1,p_o)$ and, 
in addition, for every $f\in L^1(E,\sigma)$ one has
\begin{align}\label{d-YD23-ann}
\sup_{\lambda>0}\left[\lambda\cdot
\sigma\Bigl(\Bigl\{x\in E:\int_{\Gamma_{\kappa}(x)}|(\Theta_E f)(y)|^q\,
\frac{d\mu(y)}{\delta_E(y)^{m-q\upsilon}}>\lambda^{q}\Bigr\}\Bigr)\right]
\leq C_o\|f\|_{L^1(E,\sigma)}.
\end{align}
\end{enumerate}
\end{theorem}

The conclusion~\eqref{kt-Dc-ann} in Theorem~\ref{VGds-2-ann} may be 
conveniently re-phrased by saying that the operator 
\begin{align}\label{ki-DUD-ann}
\delta_E^{\upsilon-m/q}\Theta_E:H^p(E,\rho|_{E},\sigma)
\longrightarrow L^{(p,q)}({\mathscr{X}},E)
\end{align}
is well-defined, linear and bounded, where $L^{(p,q)}(\mathscr{X},E)$ is a 
mixed norm space in the quasi-metric setting introduced in \cite{MMM} as a 
generalization of the tent spaces $T^p_q$ in ${\mathbb{R}}^{n+1}_{+}$ 
that originated with R.~Coifman, Y.~Meyer and E.~Stein in \cite{CoMeSt}
(see also \cite{BMMM} for related matters). 

\vskip 0.08in

Estimate \eqref{SF-BIS-10} in Theorem~\ref{UR-rest} now readily follows
by combining Theorem~\ref{VGds-2-ann} with \eqref{SF-BIS-ann}.

\begin{proof}[Proof of \eqref{SF-BIS-10} in Theorem~\ref{UR-rest}]
The ADR geometry of~$E$ implies that 
\begin{equation}
\left\|\Bigl(\int_{\Gamma_{\kappa}(x)}\big|(\Theta_E f)(y)\big|^2\,
\frac{d\mu(y)}{\delta_E(y)^{m-2\upsilon}}\Bigr)^{\frac{1}{2}}
\right\|_{L^2_x(E,\sigma)}^2\approx
\int_{\mathscr{X}\setminus E}\big|(\Theta_E f)(x)\big|^2 
\delta_E(x)^{2\upsilon-(m-d)}\,d\mu(x),
\end{equation}
uniformly for $f\in L^2(E,\sigma)$. Based on this and \eqref{SF-BIS-ann} 
(which has been established earlier) we conclude that the hypotheses in part 
$({\rm I})$ of Theorem~\ref{VGds-2-ann} hold with $q=2$ in the setting in which   $(d,m,\mathscr{X},E,\theta,\Theta)=(n,n+1,\mathbb{R}^{n+1},\Sigma,\nabla K, \nabla\mathcal{T})$. In turn, this yields \eqref{SF-BIS-10}, as wanted.
\end{proof}

We conclude with an extrapolation theorem that combines 
Theorems~\ref{VGds-L2XXX-ann} and~\ref{VGds-2-ann}.

\begin{theorem}\label{VGds-2.33-ann} 
Fix $\kappa>0$. If there exist $p_o\in(0,\infty)$ and a constant 
$C_o\in(0,\infty)$ such that
\begin{align}\label{GvBh-ann}
\sup_{\lambda>0}\left[\lambda\cdot
\sigma\Bigl(\Bigl\{x\in E:\int_{\Gamma_{\kappa}(x)}|(\Theta_E f)(y)|^2\,
\frac{d\mu(y)}{\delta_E(y)^{m-2\upsilon}}>\lambda^{2}\Bigr\}\Bigr)^{1/p_o}\right]
\leq C_o\|f\|_{L^{p_o}(E,\sigma)},
\end{align}
for all $f\in L^{p_o}(E,\sigma)$, then for each $p\in\bigl(\frac{d}{d+\gamma},\infty\bigr)$ there holds
\begin{align}\label{Cfrd-ann}
\left\|\Bigl(\int_{\Gamma_{\kappa}(x)}|(\Theta_E f)(y)|^2\,
\frac{d\mu(y)}{\delta_E(y)^{m-2\upsilon}}\Bigr)^{\frac{1}{2}}
\right\|_{L^p_x(E,\sigma)}\!\!\!\leq C\|f\|_{H^p(E,\rho|_{E},\sigma)},
\quad\forall\,f\in H^p(E,\rho|_{E},\sigma),
\end{align}
where $C\in(0,\infty)$ is a constant that is allowed to depend only on 
$p,C_o,\kappa,C_\theta$, and geometry. 
\end{theorem}


\begin{thebibliography}{99}
\small

\bibitem{Au} P.\,Auscher, {\it Lectures on the Kato square root problem}, 
Surveys in analysis and operator theory (Canberra, 2001), Proc. Centre Math. 
Appl. Austral. Nat. Univ. 40, Austral. Nat. Univ., Canberra, 2002, 1--18.

\bibitem{AHLMcT} P.\,Auscher, S.\,Hofmann, M.\,Lacey, A.\,McIntosh, and 
P.\,Tchamitchian, {\it The solution of the Kato Square Root Problem for second order elliptic operators on ${\mathbb{R}}^n$}, Annals of Math., 156 (2002), 633--654.

\bibitem{AS} J.\,Azzam and R.\,Schul, {\it Hard Sard: Quantitative implicit 
function and extension theorems for Lipschitz maps}, arXiv:1105.4198v3, (2012).

\bibitem{BMMM} D.\,Brigham, D.\,Mitrea, I.\,Mitrea, and M.\,Mitrea, 
{\it Triebel-Lizorkin sequence spaces are genuine mixed-norm spaces}, 
Math. Nachr., 1--15 (2012) / DOI 10.1002/mana.201100184.

\bibitem{Christ} M.\,Christ, {\it A $T(b)$ theorem with remarks on analytic 
capacity and the Cauchy integral}, Colloq. Math., 60/61 (1990), no.\,2, 601--628.

\bibitem{CMcM} R. R.\,Coifman, A.\,McIntosh and Y.\,Meyer, {\it L'int\'egrale 
de {C}auchy d\'efinit un op\'erateur born\'e sur {$L^{2}$} pour les courbes
lipschitziennes}, Ann. of Math. (2), 116 (1982), no.\,2, 361--387.

\bibitem{CoMeSt} R.\,Coifman, Y.\,Meyer, and E.M.\,Stein, {\it Some new function 
spaces and their applications to Harmonic Analysis}, Journal of Functional Analysis, 
62 (1985), 304--335.

\bibitem{CoWe77} R.R.\,Coifman and G.\,Weiss, {\it Extensions of Hardy spaces 
and their use in analysis}, Bull. Amer. Math. Soc., 83 (1977), no.\,4, 569--645.

\bibitem{David1988} G.\,David, {\it Morceaux de graphes lipschitziens et int\'egrales
singuli\`eres sur une surface}, Rev. Mat. Iberoamericana, 4 (1988), no.\,1, 73--114.

\bibitem{DaSe91} G.\,David and S.\,Semmes, {\it Singular Integrals and Rectifiable 
Sets in ${\mathbb{R}}^n$: Beyond Lipschitz Graphs}, Ast\'erisque, No.\,193, 1991.

\bibitem{DaSe93} G.\,David and S.\,Semmes, {\it Analysis of and on Uniformly 
Rectifiable Sets}, Mathematical Surveys and Monographs, AMS Series, 1993.

\bibitem{Ho} S.\,Hofmann, {\it Parabolic singular integrals of Calder\'on-type, 
rough operators and caloric layer potentials}, Duke Math. J., 90 (1997), 209--260.

\bibitem{Ho3} S.\,Hofmann, {\it Local $Tb$ Theorems and applications in PDE}, 
Proceedings of the ICM Madrid, Vol.\,II, pp.\,1375--1392, European Math. Soc., 2006.

\bibitem{HLMc} S.\,Hofmann, M.\,Lacey and A.\,McIntosh, {\it The solution of the 
Kato problem for divergence form elliptic operators with Gaussian heat kernel bounds},
Annals of Math., 156 (2002), 623--631.

\bibitem{HL} S.\,Hofmann and J.L.\,Lewis, {\it Square functions of Calder\'on 
type and applications}, Rev. Mat. Iberoamericana, 17 (2001), 1--20.

\bibitem{HMc} S.\,Hofmann and A.\,McIntosh, {\it The solution of the Kato problem 
in two dimensions}, pp.\,143--160 in ``Proceedings of the Conference on Harmonic 
Analysis and PDE" (El Escorial, 2000), Publ. Mat., Vol. extra, 2002.

\bibitem{HMc2} S.\,Hofmann and A.\,McIntosh, {\it Boundedness and applications of singular integrals and square functions: a survey}, Bull. Math. Sci., DOI 10.1007/s13373-011-0014-3.

\bibitem{MaSe79} R.A.\,Mac\'{i}as and C.\,Segovia, {\it Lipschitz functions 
on spaces of homogeneous type}, Adv. in Math., 33 (1979), 257--270. 

\bibitem{MaSe79II} R.A.\,Mac\'{\i}as and C.\,Segovia, {\it A decomposition into
atoms of distributions on spaces of homogeneous type}, Adv. in Math., 33 (1979), 
no.\,3, 271--309.

\bibitem{MMM} D.\,Mitrea, I.\,Mitrea, and M.\,Mitrea, {\it Weighted mixed-normed
spaces on spaces of homogeneous type}, preprint, (2012).

\bibitem{MMMM-G} D.\,Mitrea, I.\,Mitrea, M.\,Mitrea and S.\,Monniaux,
{\it Groupoid Metrization Theory with Applications to Analysis on Quasi-Metric Spaces 
and Functional Analysis}, Birkh\"auser, 2012.

\end{thebibliography}
\end{document}